\documentclass{amsart}

\newtheorem{theorem}{Theorem}
 
\theoremstyle{definition}
\newtheorem*{definition}{Definition}
\newtheorem{remark}{Remark}
\theoremstyle{definition}
\newtheorem{example}{Example}

\theoremstyle{plain}

\newtheorem{proposition}[theorem]{Proposition}
\newtheorem{lemma}[theorem]{Lemma}
\newtheorem{corollary}[theorem]{Corollary}

\newcommand{\Z}{\mathbf{Z}}
\newcommand{\Q}{\mathbf{Q}}

\newcommand{\Zp}{\mathbf{Z}_p}

\newcommand{\cR}{\mathcal{R}}
\newcommand{\fR}{\mathfrak{R}}
\newcommand{\disc}{\mathrm{disc}}

\DeclareMathOperator{\Haar}{Haar}

\begin{document}

\title[Probabilistic Properties of $p$-adic Polynomials]{On Certain Probabilistic Properties of Polynomials over the Ring of $p$-adic Integers}

\author{Antonio Lei and Antoine Poulin}

\maketitle
\begin{abstract}
In this article, we study several probabilistic properties of polynomials defined over the ring of $p$-adic integers under the Haar measure. {First}, we calculate the probability that a monic polynomial is separable, generalizing a result of Polak. {Second}, we introduce the notion of two polynomials being strongly coprime and calculate the probability of two monic polynomials {being} strongly coprime. Finally, we explain how our method can be used to extrapolate other probabilistic properties of polynomials over the ring of $p$-adic integers from polynomials defined over the integers modulo powers of $p$.
\end{abstract}


\section{Introduction.}
{Let $K$ be a field. We say that  a polynomial $f\in K[x]$ is \textit{separable} if its roots in an algebraic closure of $K$ are distinct. This is equivalent to $(f,f')=K[x]$, where $f'$ denotes the formal derivative of $f$. More abstractly, this is also equivalent to $A=K[x]/f$ being a separable $K$-algebra (meaning that $A$ is projective as an $A\otimes_KA$-module).  For example, $x(x+1)$ is separable, but $(x+1)^2$ is not. We may generalize this definition to commutative rings.}
Let $R$ be a commutative ring. We say that a polynomial $f\in R[x]$ is \textit{separable} if $A=R[x]/f$ is a separable $R$-algebra {(meaning that $A$ is projective as an $A\otimes_RA$-module). When $f$ is monic, this turns out to  be equivalent to $(f,f')=R[x]$ as in the case of polynomials defined over a field (we refer the readers to  \cite{AG,DI,polak} for details)}. If $f,g\in R[x]$, we say that $f$ and $g$ are \textit{relatively prime} or \textit{coprime} if there is no monic polynomial of positive degree in $R[x]$ that divides both $f$ and $g$.

The motivation for this article comes from two articles, namely \cite{polak}, where the proportion of separable monic polynomials in $(\Z/p^k\Z)[x]$ is derived from a previous result of Carlitz \cite{car} on separable polynomials in $(\Z/p\Z)[x]$, and \cite{HH}, where  certain formulae {for} the probability that two monic polynomials in $(\Z/p^k\Z)[x]$ are relatively prime are found. The results we are interested in are as follows.
\begin{theorem}[Polak, \cite{polak}]\label{thm:polak}
 Let $p$ be a prime number and  $k\geq$ 1 an integer. The proportion of monic polynomials of degree  $d\geq2$  that are separable in $(\Z/p^k\Z)[x]$ is $1-p^{-1}$. 
\end{theorem}

\begin{theorem}[Hagedorn and Hatley, \cite{HH}]\label{thm:HH}
 Let $p$ be an odd prime number and   $m,k\geq$ 1 integers. The probability that two randomly chosen monic polynomials of degrees $m$ and 2 in $(\Z/p^k\Z)[x]$ are relatively prime is given by
$$P_{\Z/p^k\Z}(m,2) = 1 - \frac{f_k(p)}{p^{3k}},$$
where $f_k(x)\in\frac{1}{2}\Z[x]$ is an explicit monic polynomial of degree $2k$.
\end{theorem}

{Results similar to Theorem~\ref{thm:HH} have been found for multiple polynomials over a finite field in \cite{BB,Cor}. Other similar problems on polynomials over finite fields can be found in \cite{Hou} and \cite{Gao}, where multivariate polynomials and common divisors of multiple polynomials are studied, respectively.}

Let  $\Zp=\varprojlim \Z/p^k\Z$ denote the ring of $p$-adic integers. We caution readers that some authors write $\Zp$ for $\Z/p\Z$. In the present article, $\Zp$ is \textit{very different} from $\Z/p\Z$. It is the ring of sequences $(a_k)_{k\ge 1}$, where $a_k\in\Z/p^k\Z$ {is} such that the image of $a_{k+1}$ under the natural projection map $\Z/p^{k+1}\Z\rightarrow \Z/p^k\Z$ equals $a_k$ for all $k\ge1$.  There is a natural bijection between $\Zp$ and the set of formal sums 
\begin{equation}\label{eq:formalsum}
\left\{\sum_{i\ge 0}c_ip^i:c_i\in\{0,1,\ldots, p-1\}\right\}.
\end{equation}
{In particular, there is a natural projection map $\Zp\rightarrow \Z/p^k\Z$ for all $k\ge1$.}
We shall see in the main part of the article that the inverse limit definition of $\Zp$ is fundamental for our theoretical proofs, whereas representing elements of $\Zp$ by formal sums is more convenient for dealing with explicit examples. 

We observe that  Polak's result, Theorem~\ref{thm:polak}, does not depend on $k$, and Hagedorn and Hatley's formula in Theorem~\ref{thm:HH} converges to 1 as $k\rightarrow \infty$.  It therefore seems natural to speculate that  the probability {that} a monic polynomial over $\Zp$ {is} separable should be $1-p^{-1}$, whereas the probability {that} two monic polynomials of degree $m$ and $2$ are {relatively prime} over $\Zp$ should be $1$.

{In order to make sense of probabilities on polynomials over $\Zp$, in this article} we consider the Haar measure on $\Zp$, which can be extended to the set of polynomials over $\Zp$ in a natural way. {This is akin to \cite{weiss}, where the distribution of splitting types and Galois groups of polynomials  over extensions of $\Zp$ are studied.} Our first result is the following generalization of Theorem~\ref{thm:polak}.
\begin{theorem}\label{thm:a}
Let $p$ be a prime number and $d\ge 2$ an integer.  With respect to the Haar measure, the probability that a degree $d$ monic polynomial  over $\Zp$ is separable is given by $1-p^{-1}$.
\end{theorem}

{This is Theorem~\ref{thm:A} below.}
{After our manuscript was first submitted to the journal, we learned that our result recovers the first part of  \cite[Theorem 1.1]{weiss} when $K_p=\Q_p$. }
In this article, we give two proofs of this theorem. {The first one is based on discriminants, which were also used by Weiss in  \cite{weiss}, and builds on results in  \cite{polak} for $ \Z /p^k\Z$.} The second proof is based on a lifting lemma for monic polynomials that we prove in {Section 2}. More specifically, we show that if $f,g\in\Zp[x]$ are monic, then the ideal generated by $f$ and $g$ {equals} $\Zp[x]$ if and only if their images under the canonical projection modulo $p$ generate $(\Z/p\Z)[x]$. We make use of the fact that a monic polynomial $f$ defined over a commutative ring $R$ is separable if and only if $(f,f')=R[x]$, where $f'$ denotes the formal derivative of $f$ ({see} \cite[\S1.4]{magid}). Therefore, our lemma allows us to translate the separability of a monic polynomial in $\Zp[x]$ to the separability  of its image in $(\Z/p\Z)[x]$. We can then apply Carlitz's result on separable polynomials over $\Z/p\Z$ to obtain a new proof of Theorem~\ref{thm:a} without using discriminants. {The second method we present here is very different from the work of Weiss. It would be interesting to see whether our lifting technique can be generalized to give an alternative proof of Weiss's result in the full generality. It would also be interesting to investigate whether our method can be used to recover other results in \cite{weiss}.}

Our lifting lemma leads us to define the following new notion on polynomials. We say that two polynomials $f$ and $g$ defined over a commutative ring $R$ are \textit{strongly coprime} if they generate $R[x]$. This is a stronger condition than being relatively prime (having no common {nonunit} factor). Our lifting lemma allows us to prove the following {theorem (which is Theorem~\ref{thm:relsep} below)}.
\begin{theorem}
Let $p$ be a prime number and $d, e\ge 1$ integers.  With respect to the Haar measure, the probability that two polynomials $f, g\in\Zp[x]$ of degree $d$ and $e${,} respectively{,} are strongly coprime is given by $1-p^{-1}$.
\end{theorem}

Finally, we show in Theorem~\ref{thm:C} that our method can be used to extrapolate the formulae of Theorem~\ref{thm:HH} to calculate the probability of monic polynomials over $\Zp[x]$ to be relatively prime.
\begin{theorem}
Let $p$ be an odd prime number and $m\geq$ 1  an integer. With respect to the Haar measure, the probability that two randomly chosen monic polynomials of degrees $m$ and 2 in $\Zp[x]$ are relatively prime is  equal to $1$.
\end{theorem}

The methods used in our proofs are inspired by Hensel's lemma, which says that the roots of a polynomial $f\in \Zp[x]$ can be found by lifting the roots of $\widetilde f\in (\Z/p\Z)[x]$ recursively, where $\widetilde f$ is the polynomial obtained from $f$ by projecting its coefficients from $\Zp$ to $\Z/p\Z$. To quote Neal Koblitz in \cite{Kob}, 

\begin{quote}
Hensel's lemma is often called the $p$-adic Newton's lemma because the approximation technique used to prove it is essentially the same as Newton's method for finding a real root of a polynomial equation with real coefficients. $\ldots$ In one respect the $p$-adic Newton's method is much better than Newton's method in the real case. In the $p$-adic case, it's guaranteed to converge to a root of a polynomial. In the real case, Newton's method usually converges, but not always. $\ldots$ Such perverse silliness is impossible in $\Q_p$.
\end{quote}

\section{Lifting of polynomials from $\Z/\MakeLowercase{p^k}\Z$ to $\Z_\MakeLowercase{p}$.}\label{S:lift}
We begin {with} the following lemma on Euclidean division for polynomials over a commutative ring (rather than a field).

\begin{lemma}[Division algorithm]\label{lem:division}
Let $R$ be a commutative ring {and} $f,g\in R[x]$ such that the leading coefficient of $g$ is a unit of $R$. Then there exist $q,r\in R[x]$ such that 
\[
f=qg+r,
\]
{where either $r = 0$ or $\deg(r) < \deg(g)$.}
\end{lemma}
\begin{proof}
Since the lemma is a simple generalization of the usual division algorithm for polynomials defined over  a field, we only give a sketch {of a} proof here.

If $\deg(f)<\deg(g)$, then we may simply take $q=0$ and $r=f$. So, we may assume that $\deg(f)\ge \deg(g)$. Let $i=\deg (f) - \deg (g) \ge0$ and write $a$ and $u$ for the leading coefficients of $f$ and $g${,} respectively. Since $u$ is a unit of $R$, there exists $b\in R$ such that $a=ub$. Then the leading terms of both $f$ and $bX^ig$ are $aX^{\deg(f)}$. In particular,
\begin{align*}
\deg (f - bX^i g) < \deg (f).
\end{align*}
If the left-hand side is smaller than $\deg(g)$, then we may take $q=bX^i$ and $r=f-bX^ig$ and we are done. Otherwise, we may repeat this procedure with $f$ replaced by $f - bX^ig$, which will produce another polynomial whose  degree is strictly smaller than that of $f- bX^ig$.  We may keep on subtracting appropriate multiples of $g$ and eventually this will produce a polynomial of degree strictly less than $\deg (g)$ {or the zero polynomial, }as required.
\end{proof}
If $f\in\Zp[x]$, we write $\widetilde{f}_k$ for its natural image in $(\Z/p^{k}\Z)[x]$ under the reduction map induced by $\Zp\rightarrow \Z/p^{k}\Z$. Note that if $f$ is monic, then so is $\widetilde{f}_k$. Furthermore,  $\deg(f) = \deg(\widetilde{f}_k)$.

\begin{theorem}\label{thm:lift}
Let $f,g\in\Zp[x]$ be two monic polynomials of degree at least $1$. Then the $\Zp[x]$-ideal generated by $f$ and $g$ equals $\Zp[x]$ if and only if the $(\Z/p\Z)[x]$-ideal generated by $\widetilde{f}_1$ and $\widetilde{g}_1$ equals $(\Z/p\Z)[x]$.
\end{theorem}
\begin{proof}
It suffices to show the following equivalence:
$$\exists \alpha,\beta \in \Zp[x], \alpha f + \beta g = 1 \Longleftrightarrow \exists \alpha_0,\beta_0 \in (\Z/p\Z)[x], \alpha_0 \widetilde{f}_1 + \beta_0 \widetilde{g}_1 = 1.$$

\noindent
($\Rightarrow$) Given $f,g,\alpha,\beta \in\Zp[x]$ such that $\alpha f+\beta g=1$, one can take  $\alpha_0$ and $ \beta_0$ to be $\widetilde{\alpha}_1$ and $ \widetilde{\beta}_1${,} respectively. Then we have trivially  $\alpha_0 \widetilde{f}_1 + \beta_0 \widetilde{g}_1 = 1$.\\ 

\noindent
($\Leftarrow$) Given $f,g \in\Zp[x]$ and $\alpha_0,\beta_0 \in(\Z/p\Z)[x]$ such that
\begin{equation}\label{eq:eqnFp}
\alpha_0 \widetilde{f}_1 + \beta_0 \widetilde{g}_1 = 1,
\end{equation}
{we} will construct a sequence $(r_i, s_i)_{i\in\Z_{\ge0}}$, where $r_i,s_i\in(\Z/p^{2^i}\Z)[x]$ such that $\deg(r_i)<\deg(g)$, $\deg(s_i) < \deg(f)${,} and 
$$r_i \widetilde{f}_{2^i} + s_i \widetilde{g}_{2^i} = 1.$$
We will construct this sequence by induction.\\

\noindent
\textbf{Base case}: Since $g$ is monic, so is $\widetilde{g}_1$  and the division algorithm of Lemma~\ref{lem:division} yields $q_0,r_0\in (\Z/p\Z)[x]$ such that $\deg(r_0) < \deg(g)$ and 
$$\alpha_0 = q_0  \widetilde{g}_1 + r_0.$$
We may rewrite \eqref{eq:eqnFp} {as}
$$r_0 \widetilde{f}_1 + s_0\widetilde{g}_1 = 1,$$
where $s_0 = q_0\widetilde{f}_1 + \beta_0$. Note that 
\begin{align*}
\deg(s_0)+\deg(g)&=\deg(s_0)+\deg(\widetilde{g}_1)\\
&=\deg(s_0\widetilde{g}_1)\\
&=\deg(r_0\widetilde{f}_1)\\
&=\deg(r_0)+\deg(f)< \deg(g)+\deg(f).
\end{align*}
Hence, we deduce that $\deg(s_0)<\deg(f)$ as required.\\

\noindent
\textbf{Inductive hypothesis}: We assume that there exist $r_i, s_i\in (\Z/p^{2^i}\Z)[x]$ such that $\deg(r_i)<\deg(g)$, $\deg(s_i) < \deg(f)${,} and 
$$r_i \widetilde{f}_{2^i} + s_i \widetilde{g}_{2^i} = 1.\\$$

\noindent
\textbf{Induction step}: Let $\hat r_i, \hat s_i\in\Zp[x] $ be two arbitrary lifts of $r_i$ and $s_i${,} respectively. Then
$$\hat r_i f + \hat s_i g = 1 + p^{2^i} Q_i$$
for certain $Q_i\in\Zp[x]$. Multiplying both sides by $1 - p^{2^i} Q_i$ yields
$$(1 - p^{2^i} Q_i)\hat r_i f + (1 - p^{2^i} Q_i)\hat s_i g = 1 - p^{2^{i+1}} Q^2_i.$$
We define $\alpha_{i+1}$ and $\beta_{i+1}$ {to} be the natural images  of $(1 - p^{2^i} Q_i)\hat r_i$ and $(1 - p^{2^i} Q_i)\hat s_i$ in $\Z/p^{2^{i+1}}\Z${,} respectively. This gives 
$$\alpha_{i+1}\widetilde{f}_{2^{i+1}}+\beta_{i+1}\widetilde{g}_{2^{i+1}} = 1$$
inside $(\Z/p^{2^{i+1}}\Z)[x]$. By the same argument as in the base case, we may replace $\alpha_{i+1}$ and $\beta_{i+1}$ by polynomials $r_{i+1}$ and $s_{i+1}$ such that $\deg(r_{i+1})<\deg(g)$ and $\deg(s_{i+1})<\deg (f)$ using the division algorithm of Lemma~\ref{lem:division}, as required.\\

Since $\Zp$ is compact and $\deg(r_i)$ and $\deg(s_i)$ are of bounded degrees, $(r_i,s_i)$ admits a subsequence that converges to a pair of polynomials $(r_\infty,s_\infty)$ in $\Zp[x]$ satisfying
\[r_\infty f+s_\infty g=1.\]
This concludes the proof of $(\Leftarrow)$.
\end{proof}
\begin{remark}\label{rk:general} Note that the same proof would go through if we replace {$\Z/p\Z$} by $\Z/p^k\Z$ for an arbitrary integer $k\ge1$. This proof gives a constructive algorithm for a linear combination in $\Zp[x]$. \end{remark}
We give an explicit example that illustrates the inductive step of the algorithm. 
\begin{example}
Take $p=5$ and consider elements of $\Z_5$ as formal sums $$c_0+c_1\cdot 5+c_2\cdot 5^2+\cdots,\quad c_i\in\{0,1,2,3,4\}$$ as in \eqref{eq:formalsum}. We consider the following monic polynomials
\begin{align*}
f&:=x^2 + (3+ 4\cdot5 + 2 \cdot 5^2 +  \cdots)x + (2 + 3\cdot5 + 4\cdot 5^2 +  \cdots),\\
g&:=x + (4 + 2\cdot  5 + 4 \cdot 5^2 +  \cdots).
\end{align*}
One can verify that, in $(\Z/5\Z)[x],$
$$\widetilde{f}_1 + (4x + 1) \widetilde{g}_1 = 1.$$
Taking the lifts $\hat r_0$ and $\hat s_0$ in $\Z_5[x]$ to be  $1 $ and $4x+1 $ respectively, we have
\begin{align*}
\hat r_0 f + \hat s_0 g& = (5+5^2 + \cdots)x + (1 + 5 + 4\cdot 5^2 + \cdots) \\
&=1 + 5\cdot (x^2+(1+5 + \cdots)x+(1 + 4\cdot 5 + \cdots)). 
\end{align*}
In particular, $ Q_0$ is given by $x^2+(1+5 + \cdots)x+(1 + 4\cdot 5 + \cdots)$.

Multiplying $\hat r_0$ and $\hat s_0$ by $ 1 - 5\cdot Q_0$ yields
\begin{align*}
(1 - 5\cdot Q_0)\hat r_0 &= (4\cdot 5+\cdots)x^2+( 4 \cdot 5 +  \cdots)x + (1 + 4\cdot5 +   \cdots),\\
(1 - 5\cdot Q_0)\hat s_0 &=(5+\cdots)x^3+( \cdots)x^2 + (4  +\cdots)x + (1 + 4\cdot5 +   \cdots)
\end{align*}
On projecting to $(\Z/5^2\Z)[x]$, we obtain
$$\alpha_1 = (4\cdot 5)x^2+ (4\cdot 5)x+ (1 + 4\cdot 5),\quad \beta_1 = 5x^3 + 4x + (1 + 4\cdot5).$$
The division algorithm of Lemma{~\ref{lem:division}} (with $R=\Z/5^2\Z$) gives
$$
\alpha_1=((4\cdot 5)x+3\cdot 5) \tilde g_2+2\cdot 5+1.
$$
Therefore, we obtain
$$r_1=2\cdot 5+1,\quad s_1=((4\cdot 5)x+3\cdot 5 ) \tilde f_2+\beta_1=(2\cdot 5+4)x+1.$$
\end{example}
\section{Probability of separable polynomials over $\Z_\MakeLowercase{p}$.}

We recall that if $R$ is a commutative ring, a polynomial $f\in R[x]$ is said to be \textit{separable} if $R[x]/f$ is a separable $R$-algebra. When $f$ is monic, this is equivalent to 
\begin{equation}
   (f,f')=R[x], \label{eq:criterion}
\end{equation}
where $f'$ denotes the formal derivative of $f$ ({see} \cite[\S1.4]{magid}). Let us write $S_R^d$ for the set of separable monic polynomials of degree $d$ in $R[x]$. Let $p$ be a fixed prime number. In \cite{polak}, Polak showed that the proportion of {separable} degree $d$ monic polynomials in $(\Z/p^k\Z)[x]$ is given by  $1-p^{-1}$ for all integers $k\ge1$ and $d\ge 2$. {In particular}, we have
\begin{equation}
\# S^d_{\Z/p^k\Z}=p^{dk}(1-p^{-1}).\label{eq:polak}
\end{equation}

The goal of this section is to generalize Polak's result to polynomials defined over the ring of $p$-adic integers $\Zp$.

Let $P_{d}(\Zp)$ denote the set of monic polynomials of degree $d$ defined over $\Zp$. In particular, if $f=X^d+a_{d-1}X^{d-1}+\cdots +a_0\in P_d(\Zp)$, we may {identify} $f$ with the $d$-tuple $(a_0,\ldots,a_{d-1})\in\Zp^d$. We equip $P_d(\Zp)=\Zp^d$ with  the product measure, denoted by $\mu_{\Haar}^d$, coming from the unique Haar measure on $\Zp$.  When $d=1$, we omit $d$ from the notation and simply write $\mu_{\Haar}$.
\begin{lemma}\label{lem:pre-image}Let $f \in (\Z/p^k\Z)[x]$ be a monic polynomial of degree $d$. Write $[f]\subset P_d(\Zp)$  for the {preimage} of $f$ under the natural projection $\Zp[x]\rightarrow (\Z/p^k\Z)[x]$. Then
$$\mu_{\Haar}^d\left([f]\right) = \frac{1}{p^{kd}}.$$
\end{lemma}
\begin{proof}
By definition, $\mu_{\Haar}(a+p^k\Zp) = \frac{1}{p^k}$ for all $a\in\Zp$. Hence, the {preimage} of any element in $\Z/p^k\Z$ in $\Zp$ has measure $\frac{1}{p^k}$. Hence the lemma follows by considering the {preimage} of each coefficient of $f$.
 \end{proof}

\begin{theorem}\label{thm:A}
Let $p$ be a prime number and $d\ge 2$ an integer. Then $$\mu^d_{\Haar}\left( S_{\Zp}^d\right)=1-p^{-1},$$
{where $S_{\Zp}^d$ denotes the set of separable monic polynomials of degree $d$ in $\Zp[x]$.}
\end{theorem}

We give two proofs. The first one is an easy generalization of the proof given in \cite{polak} using discriminants, whereas the second one is based on Theorem~\ref{thm:lift}.

\begin{proof}[Proof 1] Let us first recall from \cite[Proposition~2.1]{polak} the following result.
\begin{proposition} Let $R$ be  a  commutative  ring  with  no  nontrivial  idempotents  and  let $f\in R[x]$ be  a   monic  polynomial. Then $f$ is  separable  if  and  only  if {the discriminant }$\disc(f) \in R$ is  a  unit.
\end{proposition}
For a polynomial of a fixed degree,  $\disc(f)$  can be realized as a polynomial in the coefficients of $f$, which is independent of the  ring $R$. In particular, given any $f\in\Zp[x]$, we have $$\disc(f) \equiv \disc(\widetilde f_1) \;(\bmod\; p).$$ 

In particular, $f$ is separable in $\Zp[x]$ if and only if $\widetilde{f}_1$ is separable in $(\Z/p\Z)[x]$. Therefore,
$$S^d_{\Zp} = \bigsqcup_{{f} \in S^d_{\Z/p\Z}}[{f}],$$
where $[{f}]$ denotes the {preimage} of ${f}$ in $\Zp[x]$ under the natural projection  {and $\sqcup$ denotes the disjoint union of sets}. By Lemma~\ref{lem:pre-image} and \eqref{eq:polak}, we deduce that
\begin{align*}
\mu_{\Haar}^d\left(S^d_{\Zp}\right)& = \mu_{\Haar}^d\left(\bigsqcup_{{f} \in S^d_{\Z/p\Z}}[{f}]\right)\\
&=\sum_{{f} \in S^d_{\Z/p\Z}}\mu_{\Haar}^d([{f}])\\
&=\frac{\# S^d_{\Z/p\Z}}{p^d}\\
&=\frac{p^d(1-p^{-1})}{p^d}\\
&=1-p^{-1},
\end{align*}
as required. 
\end{proof}
\begin{proof}[Proof 2]
 Let $f\in \Zp[x]$. Recall that its natural image in $(\Z/p\Z)[x]$ is denoted by $\widetilde{f}_1$. Theorem \ref{thm:lift} tells us that
$$(f,f') = \Zp[x] \Leftrightarrow (\widetilde{f}_1,\widetilde{f}_1') = (\Z/p\Z)[x].$$
{Thus, by the criterion of separability given in \eqref{eq:criterion}},  $f$ is separable in $\Zp[x]$ if and only if $\widetilde{f}_1$ is separable in $(\Z/p\Z)[x]$. {As in Proof 1, we may now deduce Theorem~\ref{thm:A} from  the fact  that $ \# S^d_{\Z/p\Z}=p(1-p^{-1})$ as given by \eqref{eq:polak}, which is a special case of the main result of \cite{polak} (see Theorem~\ref{thm:polak}). }
\end{proof}

\section{Probability of strongly coprime polynomials over $\Z_\MakeLowercase{p}$.}

If  $K$ is a field, then a polynomial  $ f\in K[x]$ is separable if and only if
$$(f,f')=K[x].$$
There is a similar notion for relatively prime polynomials. That is, if $f,g \in K[x]$, then $f$ and $g$ are relatively prime if and only if
$$(f,g) = K[x].$$

However, if we replace $K$ by a commutative ring $R$, then it is possible for two relatively prime polynomials $f,g\in R[x]$ {to be} such that $f$ and $g$ generate a proper ideal of $R[x]$. We define the following new notion that allows us to extrapolate information on relatively prime polynomials in $(\Z/p\Z)[x]$ to polynomials in $\Zp[x]$.

\begin{definition} \label{def:sep}
Let $R$ be a commutative ring. We say that two polynomials $f,g\in R[x]$ are \textbf{strongly coprime} if 
$$(f,g) = R[x].$$
\end{definition}
\begin{theorem}\label{thm:relsep}
Let $p$ be a prime number and $d, e\ge 1$ integers.  With respect to the Haar measures on $P_d(\Zp)$ and $P_e(\Zp)$, the probability that two random polynomials in $P_d(\Zp)$ and $  P_e(\Zp)$ are strongly coprime is given by $1-p^{-1}$.
\end{theorem}
\begin{proof}
Given a ring $R$, let us write $\cR^{d,e}_R$ for the set of pairs of relatively prime polynomials $(f,g)$, where $f,g\in R[x]$ such that $\deg(f)=d$ and $\deg(g)=e$.

Recall from Theorem~\ref{thm:lift} that 
$$(f,g) = \Zp[x] \Leftrightarrow (\widetilde{f}_1,\widetilde{g}_1) =( \Z/p\Z)[x].$$
In other words, $f$ and $g$ are strongly coprime {if and} only if their projections are. Hence,
$$\cR^{d,e}_{\Zp} = \bigsqcup_{(f,g) \in \cR^{d,e}_{\Z/p\Z}}[f]\times[g].$$
Let $\mu_{\Haar}^{d,e} := \mu_{\Haar}^d{{\times}}\mu_{\Haar}^e$ be the product measure on $P_d(\Zp)\times P_e(\Zp)$. Then, Lemma~\ref{lem:pre-image} tells us that
\begin{align*}
\mu_{\Haar}^{d,e}(\cR^{d,e}_{\Zp})& =\sum_{(f,g) \in \cR^{d,e}_{\Z/p\Z}}\mu_{\Haar}^{d,e}\left([f]\times[g]\right)\\
&=  \frac{\#\cR^{d,e}_{\Z/p\Z}}{p^{d+e}}.
\end{align*}
It {was} proved in \cite{BB} that $$\#\cR^{d,e}_{\Z/p\Z}=p^{d+e}(1-p^{-1}){,}$$
{hence the result follows}.
\end{proof}
Note that we may deduce a similar statement {for} polynomials defined over  $\Z/p^k\Z$ for any integers $k\ge1$.
\begin{corollary}
Let $p$ be a prime number and $d,e, k\ge 1$  integers. The probability that two monic polynomials of degree $d$ and $e$ in $\Z/p^k\Z$[x] are strongly coprime is given by $1-p^{-1}$.
\end{corollary}
\begin{proof}
As explained in Remark~\ref{rk:general}, we have the equivalence 
$$(f,g) = \Zp[x] \Leftrightarrow (\widetilde{f}_k,\widetilde{g}_k) = (\Z/p^k\Z)[x]\Leftrightarrow (\widetilde{f}_1,\widetilde{g}_1) = (\Z/p\Z)[x].$$
Hence, the calculation in the proof of Theorem~\ref{thm:relsep} goes through if we replace $\Zp$ by $\Z/p^k\Z$.
\end{proof}

\section{Probability of relatively prime polynomials over $\Z_\MakeLowercase{p}$.}
{For a} polynomial over $\Zp${,} being square-free is weaker than being separable. In fact, as remarked in \cite{weiss}, the set of square-free polynomials over $\Zp$ has full Haar measure, whereas that of separable {polynomials} has measure $1-p^{-1}$ by Theorem~\ref{thm:A}. Similarly, two polynomials over $\Zp$ being relatively prime polynomials (i.e.{,} having no common {nonunit} factor) is weaker than being strongly coprime.  In this section we illustrate how a similar calculation to what we made in the previous sections allows us to  extrapolate Hagedorn and Hatley's formulae in Theorem~\ref{thm:HH} to a result on relatively prime polynomials over $\Zp$.
\begin{theorem}\label{thm:C}
Let $p$ be an odd prime number and   $m\geq1$   an integer. With respect to the Haar measure, the probability that two randomly chosen monic polynomials in $P_m(\Zp)$ and $P_2(\Zp)$ are relatively prime is $1$.
\end{theorem}
\begin{proof}
Let $f,g\in \Zp[x]$. Observe that if $\widetilde{f}_k$ and $\widetilde{g}_k$ are relatively prime as polynomials in $(\Z/p^k\Z)[x]$, then $f$ and $g$ are also relatively prime since any common {nonunit} factors between them would give common {nonunit} factors of $\widetilde{f}_k$ and $\widetilde{g}_k$.

For $R=\Zp$ or $\Z/p^k\Z$, let us write  $\fR^{m}_{R}$ for the set of relatively prime polynomials $(f,g)$ with $\deg(f)=m$ and $\deg(g)=2$. {Recall that, by our notation, the measure on $\fR^{m}_{\Z_p}$ will be $\mu_{\Haar}^{m,2}$}.Then our observation above tells us that
$$\fR^{m}_{\Zp} \supset \bigsqcup_{(f,g) \in \fR^{m}_{\Z/p^k\Z}}[f]\times[g].$$
Therefore, on applying Lemma~\ref{lem:pre-image} and Theorem~\ref{thm:HH}, we obtain the following inequality{:}
\begin{align*}
\mu_{\Haar}^{m,2}\left(\fR^{m}_{\Zp}\right)& \geq \mu_{\Haar}^{m,2}\left(\bigsqcup_{(f,g) \in \fR^{m}_{\Z/p^k\Z}}[f]\times[g]\right)\\
&= \frac{\#\fR^{m}_{\Z/p^k\Z}}{p^{(2+m)k}}\\
&=1-\frac{f_k(p)}{p^{3k}}.
\end{align*}
From \cite[p. 224]{HH} we have that the polynomial $f_k$ is of degree $2k$ and its coefficients have absolute value at most $2$. Therefore,  $$1-\frac{f_k(p)}{p^{3k}}\rightarrow 1,\quad \text{as }k\rightarrow \infty.$$ But $\mu_{\Haar}^{m,2}\left(\fR^{m}_{\Zp}\right)$ is at most $1$ by definition. This forces
$$\mu_{\Haar}^{m,2}\left(\fR^{m}_{\Zp}\right)=1,$$
as required.
\end{proof}

To conclude, we would like to outline a number of future problems that may be studied using our method.
\begin{itemize}
    \item[(i)] We may calculate the probability of two polynomials over $\Zp$ of any degrees to be relatively prime by estimating the number of polynomials over $\Z/p^k\Z$ of the same degrees that are relatively prime.
    \item[(ii)]  On generalizing techniques of \cite{HH}, we may generalize results of \cite{BB,Cor,Gao} to multiple, potentially multivariate,  polynomials over $\Z/p^k\Z$ (instead of fields). Once this is achieved, we may then apply our method to study multiple, potentially multivariate, polynomials over $\Zp$.
    \item[(iii)] We may, as in \cite{weiss}, consider the ring of integers of a finite extension of $\Q_p$ and generalize our results on coprime and strongly coprime polynomials.
    \item[(iv)] We may give a new proof of Weiss's formula on the distribution  of splitting types of irreducible polynomials defined over  the ring of integers of a finite extension of $\Q_p$.
    \item[(v)] It would also be interesting to investigate whether we may refine the different estimates on the distribution of Galois groups of irreducible polynomials  in \cite[Theorems 1.2--1.6]{weiss} using our lifting technique.
\end{itemize}

\subsection*{Acknowledgments.}
We are grateful to Jeffrey Hatley and Benjamin Weiss for answering our questions. We thank Hugues Bellemare and Gautier Ponsinet for  helpful discussions. Finally, we thank the anonymous referees and the editor for many useful comments and suggestions on earlier versions of this article. The first named author's research is supported by the NSERC Discovery Grants Program 05710. Parts of this work were carried out during the second named author's research internship at Laval University in the summer of 2018; he was supported by a NSERC Undergraduate Student Research Award during the internship.

\end{document}